\documentclass[10pt]{article}
\usepackage[dvips]{color}
\usepackage{epsfig}
\usepackage{float,amsthm,amssymb,amsfonts}
\usepackage{ amssymb,amsmath,graphicx, amsfonts, latexsym}
\begin{document}
\theoremstyle{plain}
\newtheorem{theorem}{{\bf Theorem}}[section]
\newtheorem{corollary}[theorem]{Corollary}
\newtheorem{lemma}[theorem]{Lemma}
\newtheorem{proposition}[theorem]{Proposition}
\newtheorem{remark}[theorem]{Remark}

\theoremstyle{definition}
\newtheorem{defn}{Definition}
\newtheorem{definition}[theorem]{Definition}
\newtheorem{example}[theorem]{Example}
\newtheorem{conjecture}[theorem]{Conjecture}
\def\im{\mathop{\rm Im}\nolimits}
\def\dom{\mathop{\rm Dom}\nolimits}
\def\rank{\mathop{\rm rank}\nolimits}
\def\nullset{\mbox{\O}}

\def\implies{\; \Longrightarrow \;}

\def\GR{{\cal R}}
\def\GL{{\cal L}}
\def\GH{{\cal H}}
\def\GD{{\cal D}}
\def\GJ{{\cal J}}

\def\set#1{\{ #1\} }
\def\z{\set{0}}
\def\Sing{{\rm Sing}_n}
\def\nullset{\mbox{\O}}

\title{On certain semigroups of full
contractions  of a finite chain }
\author{\bf  A. Umar and M. M. Zubairu \footnote{Corresponding Author. ~~Email: $aumar@pi.ac.ae$} \\[3mm]
\it\small Department of Mathematics, The Petroleum Institute, Sas Nakhl,\\
\it\small Khalifa University of Science and Technology, P. O. Box 2533, Abu Dhabi, UAE\\
\it\small  \texttt{aumar@pi.ac.ae}\\[3mm]
\it\small  Department of Mathematics, Bayero  University Kano, P. M. B. 3011, Kano, Nigeria\\
\it\small  \texttt{mmzubairu.mth@buk.edu.ng}\\
}
\date{\today}
\maketitle\

\begin{abstract}
 Let $[n]=\{1,2,\ldots,n\}$ be a finite chain and let  $\mathcal{T}_{n}$  be the semigroup of full transformations on $[n]$.  Let $\mathcal{CT}_{n}=\{\alpha\in \mathcal{T}_{n}: (for ~all~x,y\in [n])~\left|x\alpha-y\alpha\right|\leq\left|x-y\right|\}$, then $\mathcal{CT}_{n}$ is a  subsemigroup of $\mathcal{T}_{n}$. In this paper, we give a necessary and sufficient condition for an element  to be regular and characterize all the Green's equivalences for the semigroup $\mathcal{CT}_{n}$. We further show that the semigroup $\mathcal{CT}_{n}$ is left abundant semigroup.
 \end{abstract}

\section{Introduction}
 Let $[n]=\{1,2, \ldots ,n\}$ be a finite chain, a map $\alpha$ which has domain and range both subset of $[n]$ is said to be a \emph{transformation}.  A \emph{transformation} $\alpha$ whose domain is a subset of $[n]$  (i. e., $\dom~\alpha \subseteq [n]$) is said to be  \emph{partial}. The collection of all partial transformations of $[n]$ is known as semigroup of partial transformations, usually denoted by $\mathcal{P}_{n}$. A partial transformations whose domain is the whole $[n]$ (i. e., $\dom~\alpha = [n]$) is said to \emph{full}(or \emph{total}). The collection of all full (or total) transformations of $[n]$ is known as full transformation semigroup.  A map $\alpha\in \mathcal{T}_{n}$ is said to be \emph{order preserving} (resp., \emph{order reversing}) if  (for all $x,y \in [n]$) $x\leq y$ implies $x\alpha\leq y\alpha$ (resp., $x\alpha\geq y\alpha$); is \emph{order decreasing} if (for all $x\in [n]$) $x\alpha\leq x$;  an \emph{isometry} (i. e., \emph{ distance preserving}) if (for all $x,y \in [n]$) $|x\alpha-y\alpha|=|x-y|$;   a \emph{contraction} if (for all $x,y \in [n]$) $|x\alpha-y\alpha|\leq |x-y|$. Let $\mathcal{CT}_{n}=\{\alpha\in \mathcal{T}_{n}:(for ~all~x,y\in [n])~ |x\alpha-y\alpha|\leq |x-y|\}$ and $\mathcal{OCT}_{n}=\{\alpha\in \mathcal{CT}_{n}: (for ~all~x,y\in [n])~x\leq y ~ implies ~ x\alpha\leq y\alpha\}$. Then $\mathcal{CT}_{n}$  and $\mathcal{OCT}_{n}$ are  subsemigroups of $\mathcal{T}_{n}$ and  $\mathcal{{OT}}_{n}$ (where $\mathcal{{OT}}_{n}$ denotes the semigroup of order preserving full transformations of $[n]$),  respectively. They are known as semigroups of full  contractions  and order preserving partial contractions of $[n]$ respectively. The study of these semigroups was initiated in 2013 by Umar and Alkharousi \cite{af}  supported by a grant from The Research Council of Oman (TRC). In the proposal \cite{af}, notations for the semigroups and their subsemigroups were given, as such we maintain the same notations in this paper. For standard concepts in semigroup theory, we refer the reader to Howie  \cite{Howie1} and Higgins \cite{ph}.

 Let $S$ be a semigroup and $a,b\in S$. If $S^{1}a=S^{1}b$ (i. e., $a$ and $b$ generate the same principal left ideal) then we say that $a$ and $b$ are related by $\mathcal{L}$ and we write $(a,b)\in \mathcal{L}$ or $a\mathcal{L}b$, if $aS^{1}=bS^{1}$ (i. e., $a$ and $b$ generate the same principal right ideal) then we say $a$ and $b$ are related by $\mathcal{R}$ and we write $(a,b)\in \mathcal{R}$ or $a\mathcal{R}b$ and if $S^{1}aS^{1}=S^{1}bS^{1}$ (i. e., $a$ and $b$ generate the same principal two sided ideal) then we say $a$ and $b$ are related by $\mathcal{J}$ and we write $(a,b)\in \mathcal{J}$ or $a\mathcal{J}b$. Each of the relations $\mathcal{L}$, $\mathcal{R}$ and $\mathcal{J}$ is an equivalence on $S$. The relations $\mathcal{H}= \mathcal{L}\cap \mathcal{R}$ and $\mathcal{D}=\mathcal{L}\circ \mathcal{R}$ are all equivalences on $S$. These five equivalences known as Green's relations, were first introduced by J. A. Green in 1951 \cite{gr}.

  The Green's relations for the semigroup $\mathcal{CP}_{n}$ and some of its  subsemigroups have been studied by  \cite{mmz, az, py}.  In 2017, Garba \emph{et al.} \cite{garbac} characterized Green's and  starred Green's relations for the semigroup $\mathcal{CT}_{n}$. However, their characterization of Green's $\mathcal{L}$ relation is incomplete, which  also affects the $\mathcal{D}$ relation characterization. In 2018, Ali \emph{et al.} characterized Green's relations for the more general semigroup $\mathcal{CP}_{n}$, and results concerning Green's relations and regularity of elements in some of its subsemigroups can be deduced from the results obtained in that paper.

  In this paper we obtain complete characterizations of Green's relations for the semigroup of full contractions $\mathcal{CT}_{n}$. We deduce corresponding results for regularity of elements in the semigroup $\mathcal{CT}_{n}$ and its subsemigroups of order preserving full contractions and of order preserving or order reversing  full contractions,  $\mathcal{OCT}_{n}$ and  $\mathcal{ORCT}_{n}$, respectively, from Theorem~\eqref{reeg}. We further show that the semigroup $\mathcal{CT}_{n}$ is \emph{left abundant} but not \emph{right abundant} (for $n\geq 4$) and explore some \emph{orthodox} subsemigroups of $\mathcal{CT}_{n}$ and its {\emph Rees factor} semigroups. This paper is therefore a natural analogue of Umar and Zubairu \cite{az}. We now introduce some basic definitions to achieve our objective.

 Let $\alpha$ be element of $\mathcal{CP}_{n}$. Let  $\dom\ \alpha$, $\im~\alpha$ and  $h~(\alpha)$  denote, the domain of $\alpha$, image  of $\alpha$ and $|\im~\alpha|$, respectively. For $\alpha,\beta \in \mathcal{CT}_{n}$,  the composition of $\alpha$ and $\beta$ is defined as $x(\alpha \circ \beta) =((x)\alpha)\beta$ for any $x$ in $[n]$.  Without ambiguity, we shall be using the notation $\alpha\beta$ to denote $\alpha \circ\beta$.

   Let
    \begin{equation}\label{1}
    \alpha=\left( \begin{array}{cccc}
                           A_{1} & A_{2} & \ldots & A_{p} \\
                           x_{1} & x_{2} & \ldots & x_{p}
                         \end{array}
   \right)\in \mathcal{T}_{n}  ~~  (1\leq p\leq n),
    \end{equation}

     where, $x_{i}\alpha^{-1}=A_{i}$ ($1\leq i\leq p$) are equivalence classes under the relation $ker\alpha=\{(x,y)\in n\times n: x\alpha=y\alpha\}$. The collection of all the equivalence classes of the relation $ker\alpha$, is the partition of the domain of  $\alpha$, and is denoted by $\textbf{Ker}~\alpha$, i. e., $\textbf{Ker}~\alpha=\{A_{1}, A_{2}, \ldots A_{p}\}$ and $\dom~\alpha=A_{1}\cup A_{2}\cup\ldots A_{p}$ where $(p\leq n)$. A subset $T_{\alpha}$ of $[n]$ is said to be a \emph{transversal} of the partition $\textbf{Ker}~\alpha$ if $|T_{\alpha}|=p$, and $|A_{i}\cap T_{\alpha}|=1$ ($1\leq i\leq p$). A transversal  $T_{\alpha}$  is said to be \emph{relatively convex} if for all $x,y\in T_{\alpha}$ with $x\leq y$ and if $x\leq z\leq y$ ($z\in \dom~\alpha$), then $z\in T_{\alpha}$. Notice that every convex transversal is necessarily relatively convex but not \emph{vice-versa}.

A transversal $T_{\alpha}$ is said to be \emph{admissible} if and only if the map $A_{i}\mapsto t_{i}$  ($t_{i}\in T_{\alpha},\, i\in\{1,2,\ldots,p\}$) is a contraction. Notice that every (relatively) convex transversal is admissible but not \emph{vice-versa},  see \cite{mmz}.

     A map $\alpha\in \mathcal{T}_{n}$ is said to be an \emph{isometry} if and only if $|x\alpha-y\alpha|=|x-y|$ for all $x,y\in \dom~\alpha$. If we consider $\alpha$ as expressed in \eqref{1}, then $\alpha$ is an isometry if and only if $|x_{i}-x_{j}|=|a_{i}-a_{j}|$ for all $a_{i}\in A_{i}$ and $a_{j}\in A_{j}$ ($i,j\in\{1,2,\ldots,p\}$). Notice that this forces the blocks $A_{i}$ ($i=1,\ldots,p$) to be singletons, because $\alpha$ is one$-to-$one. In other words, $\alpha$ is an isometry if and only if $\dom~{\alpha}=\{a_{i}: 1\leq i\leq p\}=\{x_{i} + e: 1\leq i\leq p\} = (\dom~{\alpha})\alpha + e$ (called a \emph{translation}) or $\dom~{\alpha}=\{a_{i}: 1\leq i\leq p\}=\{x_{p-i+1} + e: 1\leq i\leq p\} = (\dom~{\alpha})\alpha + e$  (called a \emph{reflection}) for some integer $e$. An element $a$ in a semigroup $S$ is said to be an idempotent if and only if $a^{2}=a$. It is well known that  an element $\alpha\in \mathcal{P}_{n}$  is  an idempotent if and only if $\im~\alpha=F(\alpha)=\{x\in \dom~\alpha: x\alpha =x\}$. Equivalently, $\alpha$ is an idempotent if and only if $x_{i}\in A_{i}$ for  $1\leq i\leq p$, that is to say that the \emph{blocks} $A_i$ are \emph{stationary}.

     An element $a\in S$ is said to be \emph{regular} if there exists $b\in S$ such that $a=aba$. The following results from \cite{mmz} will be useful in the proofs of some of the results in the subsequent sections.

     \begin{theorem}[\cite{mmz}, Theorem 2.1]\label{reeg} Let $\alpha\in \mathcal{CP}_{n}$ be  as expressed in \eqref{1}. Then $\alpha$ is regular if and only if  there exists an admissible transversal $T_{\alpha}$ of $\textbf{Ker}~\alpha$ such that $\left|t_{j}-t_{i}\right|=\left|t_{j}\alpha-t_{i}\alpha\right|$ for all $t_{j}, t_{i}\in T_{\alpha}$ ($i,j\in \{1,2,\ldots,p\}$). Equivalently, $\alpha$ in $\mathcal{CP}_{n}$ is regular if and only if there exists an admissible transversal $T_{\alpha}$, such that the map $t_{i}\mapsto x_{i}$  ($i\in\{1,2,\ldots, p\}$) is an isometry.
\end{theorem}

\begin{lemma}[\cite{mmz}, Theorem 3.1]\label{bb} For every  $\alpha\in \mathcal{CP}_{n}$, the partition $\textbf{Ker}~\alpha$   have   a maximum finer partition say $\textbf{Ker}~\gamma$ (for some $\gamma\in \mathcal{CP}_{n}$) with  convex transversals.
\end{lemma}
\begin{lemma}[\cite{mmz} Lemma 1.4] Let           $\alpha\in \mathcal{ORCP}_{n}$ be as expressed in \eqref{1} where   $\textbf{Ker}~\alpha=\{A_{1}<A_{2}<\ldots<A_{p}\}$  (for $n\geq 4$). If there exists $k\in \{2,\ldots,p-1\}$  such that $|A_{k}|\geq 2$ then $\textbf{Ker}~\alpha$  have no convex transversal.
\end{lemma}

\begin{corollary}\label{d1} Let    $\alpha\in \mathcal{OCRT}_{n}$   and    $\textbf{Ker}~\alpha=\{A_{1}<A_{2}<\ldots<A_{p}\}$  (for $n\geq 4$). If there exists $k\in \{2,\ldots,p-1\}$  such that $|A_{k}|\geq 2$ then $\textbf{Ker}~\alpha$  have no convex transversal.
\end{corollary}

\begin{corollary}

 Let    $\alpha$ be in  $\mathcal{CP}_{n}$ (resp., $\mathcal{CT}_{n}$)   $\mathcal{ORCP}_{n}$ (resp., $\mathcal{ORCT}_{n}$),    and    $\textbf{Ker}~\alpha=\{A_{1}<A_{2}<\ldots<A_{p}\}$  (for $n\geq 4$).  If    $|A_{k}|=1$  for all  $k\in \{2,\ldots,p-1\}$    then $\textbf{Ker}~\alpha$  have a convex transversal.

\end{corollary}

The following lemma is from \cite{adu}.
 \begin{lemma}[\cite{adu}, Lemma 1.2]\label{tc}  Let $\alpha\in \mathcal{CT}_{n}$ and let $|\im~\alpha|=p$. Then $\im~\alpha$ is convex.
\end{lemma}

     Now, let
     \begin{equation}\label{2} \alpha=\left(\begin{array}{cccc}
                                                                            A_{1} & A_{2} & \ldots & A_{p} \\
                                                                            x_{1} & x_{2} & \ldots & x_{p}
                                                                          \end{array}
\right)~~ and~~ \beta=\left(\begin{array}{cccc}
                                                                            B_{1} & B_{2} & \ldots & B_{p} \\
                                                                            y_{1} & y_{2} & \ldots & y_{p}
                                                                          \end{array}
\right)\in {\cal CP}_n~~ (p\leq n).\end{equation}
 The following result itemized as (i), (ii) and (iii) are Theorems 2.4, 2.5 and 2.6, respectively in \cite{mmz}.

\begin{theorem}[\cite{mmz}, Theorems 2.3, 2.5 $\mbox{\&}$ 2.6]\label{14} Let $\alpha, \beta\in \mathcal{CP}_{n}$ be as expressed in \eqref{2}. Then
\begin{itemize}

\item[(i)] $(\alpha, \beta)\in \mathcal{L}$ if and only if   $\textbf{Ker}~\alpha$ and $\textbf{Ker}~\beta$ have  admissible finer partitions,  $\textbf{Ker}~\gamma_{1}$ and $\textbf{Ker}~\gamma_{2}$ (for some $\gamma_{1}$ and $\gamma_{2}$ in $\mathcal{CP}_{n}$), respectively, such that there exists either a translation $\tau_{i}\mapsto \sigma_{i}$ and $\tau_{i}\alpha= \sigma_{i}\beta$ or a reflection $\tau_{i}\mapsto \sigma_{s-i+1}$ and $\tau_{i}\alpha= \sigma_{s-i+1}\beta$ for all $i=1,\ldots,s$ ($s\geq p$), where $A_{\alpha}=\{\tau_{1}, \ldots,\tau_{s}\}$ and $B_{\alpha}=\{\sigma_{1}, \ldots,\sigma_{s}\}$, are  the admissible transversals of $\textbf{Ker}~\gamma_{1}$ and $\textbf{Ker}~\gamma_{2}$, respectively.
\item[(ii)]$(\alpha, \beta)\in \mathcal{R}$ if and only if $ker\alpha=ker\beta$ and there exists either a translation $x_{i}\mapsto y_{i}$ or a reflection $x_{i}\mapsto y_{p-i+1}$ ($1\leq i\leq p$).

\item[(iii)] $(\alpha,\beta)\in \mathcal{D}$ if and only if there exist isometries $\vartheta_{1}$ and $\vartheta_{2}$ from $\textbf{Ker}~\gamma_{1}$ to $\textbf{Ker}~\gamma_{2}$ and  from $\im~\alpha$ to $\im~\beta$, where $\textbf{Ker}~\gamma_{1}$ and $\textbf{Ker}~\gamma_{2}$ (for some $\gamma_{1}$ and $\gamma_{2}$ in $\mathcal{CP}_{n}$) are maximum admissible finer partitions of  $\textbf{Ker}~\alpha$ and $\textbf{Ker}~\beta$, respectively.
\end{itemize}
\end{theorem}

 We pause here to describe the shortcomings of the characterization of Green's $\mathcal{L}$ relation on $\mathcal{CT}_{n}$ in Garba \emph{et al.} \cite{garbac}. The result stated as part of (\cite{garbac}, Theorem 2.2) with respect to the characterization of the Green's $\mathcal{L}$ relation  is incomplete, because the elements    $\alpha=\left( \begin{array}{ccccc}
                            1 & \{2,3\} & \{4,6\} & 5 \\
                            {1} & {2}  & 3 & 4
                           \end{array}
\right)$ and\\ $\beta=\left( \begin{array}{ccccc}
                            {1} & 2 & \{3,4, 6\} & 5 \\
                            {4} & {3}  & 2 & 1
                           \end{array}
\right)\in \mathcal{CT}_{6}$ have no convex transversals and therefore (\cite{garbac}, Theorem 2.2) does not say whether the two elements are $\mathcal{L}$ related or not. However, they are indeed $\mathcal{L}$ related as we shall see from Corollary \eqref{tt}.

A partition $\textbf{Ker}~\gamma$ (for $\gamma \in \mathcal{P}_{n}$) is said to be a \emph{refinement} of the partition $\textbf{Ker}~\alpha$ if $\ker~\gamma\subseteq \ker~\alpha$. Thus, if $\textbf{Ker}~\gamma=\{A_{1}^{'}, A_{2}^{'}, \ldots, A_{s}^{'}\}$ and $\textbf{Ker}~\alpha=\{A_{1}, A_{2}, \ldots, A_{p}\}$ then for any $1\leq i\leq p$, $A_{i}={\overset{}{\underset{j}\cup}}A^{'}_{j}$ for some $1\leq j\leq s$ and thus $p\leq s.$
A refined partition $\textbf{Ker}~\gamma$ of $\textbf{Ker}~\alpha$ is said to be\emph{ maximum}, if  $\ker~\gamma\subseteq \ker~\alpha$ and every refined relation of $\ker~\alpha$ say $\ker~\theta$ is contained in $\ker~\gamma$. Moreover, if there are at least two maximal relations say $\ker~\tau_{i}$ ($i\geq 2$)  contained in $\ker~\alpha$, then $\ker~\gamma$ is maximum if $\ker~\gamma={\overset{}{\underset{i\geq 2}\cap}}\ker~\tau_{i}.$  A refined partition $\textbf{Ker}~\gamma$ of $\textbf{Ker}~\alpha$ is admissible if it has an admissible transversal. The following lemma is from \cite{mmz}.

\begin{lemma}[\cite{mmz}, Lemma 3.1]\label{bkk} For every $\alpha\in \mathcal{CP}_{n}$,  $\textbf{Ker}~\alpha$   has   a maximum finer partition say $\textbf{Ker}~\gamma$ (for some $\gamma\in \mathcal{CP}_{n}$) with an admissible transversal.

\end{lemma}

Notice that if $\gamma\in \mathcal{CT}_{n}$ then from Lemma~\eqref{tc}, we see that $\im~\gamma$ is convex i. e., $A_{\gamma}\gamma=\{\tau_{1}, \ldots,\tau_{s}\}\gamma$ of $\textbf{Ker}~\gamma$ is convex. In view of this we give  the following definition. A transformation $\alpha\in\mathcal{CT}_{n}$, is said to have a \emph{refined convex} partition $\textbf{Ker}~\gamma$ if $\gamma$ is a refinement of $\alpha$ and $\textbf{Ker}~\gamma$ has a convex transversal. Thus we have as a corollary to Lemma~\eqref{bkk}, the following:

\begin{corollary}\label{bb2} For every  $\alpha\in \mathcal{CT}_{n}$, the partition $\textbf{Ker}~\alpha$   has  a maximum finer partition say $\textbf{Ker}~\gamma$ (for some $\gamma\in \mathcal{CT}_{n}$) with a convex transversal.
\end{corollary}

 We now deduce the characterization of Green's equivalences on $\mathcal{CT}_{n}$ from Theorem~\eqref{14}.
\begin{corollary}\label{tt} Let $\alpha,\beta\in \mathcal{CT}_{n}$  be as expressed in \eqref{2}. Then
 \begin{itemize}
  \item[(i)]  $(\alpha, \beta)\in \mathcal{L}$ if and only if  $\textbf{Ker}~\alpha$ and $\textbf{Ker}~\beta$ have convex finer partitions,  $\textbf{Ker}~\gamma_{1}$ and $\textbf{Ker}~\gamma_{2}$ (for some $\gamma_{1}$ and $\gamma_{2}$ in $\mathcal{CT}_{n}$), respectively, such that there exists either a translation $\tau_{i}\mapsto \sigma_{i}$ and $\tau_{i}\alpha= \sigma_{i}\beta$ or a reflection $\tau_{i}\mapsto \sigma_{s-i+1}$ and $\tau_{i}\alpha= \sigma_{s-i+1}\beta$ for all $i=1,\ldots,s$ ($s\geq p$), where $A_{\alpha}=\{\tau_{1}, \ldots,\tau_{s}\}$ and $B_{\alpha}=\{\sigma_{1}, \ldots,\sigma_{s}\}$, are  the convex transversals of $\textbf{Ker}~\gamma_{1}$ and $\textbf{Ker}~\gamma_{2}$, respectively;
  \item[(ii)] $(\alpha,~\beta)\in \mathcal{R}$ if and only  if $ker~\alpha=ker~\beta$;
  \item[(iii)] $(\alpha,~\beta)\in \mathcal{D}$ if and only   if there exist isometries $\vartheta_{1}$ and $\vartheta_{2}$ from $\textbf{Ker}~\gamma_{1}$ to $\textbf{Ker}~\gamma_{2}$ and  from $\im~\alpha$ to $\im~\beta$, respectively.
\end{itemize}

 \end{corollary}
\begin{proof}
The proof follows directly from Theorem~\eqref{14} coupled with the fact that $\dom\alpha=[n]$ and every admissible transversal is necessarily convex, by Lemma~\eqref{tc}.
\end{proof}

Next, let us exemplify the above theorem using the counter example we stated earlier in the beginning of this section.

\begin{example} Consider
 $\alpha=\left( \begin{array}{ccccc}
                            1 & \{2,3\} & \{4,6\} & 5 \\
                            {1} & {2}  & 3 & 4
                           \end{array}
\right)$ and $\beta=\left( \begin{array}{ccccc}
                            {1} & 2 & \{3,4, 6\} & 5 \\
                            {4} & {3}  & 2 & 1
                           \end{array}
\right)$ in $\mathcal{CT}_{6}$. Notice that the partition $\textbf{Ker}~\gamma=\{\{1\}, \{2\},\{3\}, \{4,6\}, \{5\}\}$ (for some $\gamma$) is a refined partition of both the partitions $\textbf{Ker}~\alpha$ and $\textbf{Ker}~\beta$. Moreover, it's admissible since it has a convex transversal $\{1,2,3,4,5\}$. Now define a map from $\textbf{Ker}~\gamma$ to $\{1,2,3,4,5\}$ by $\delta=\left( \begin{array}{cccccc}
                            1 & 2 &  3 & \{4,6\} &5 \\
                            5 & 4  & 3 &2 & 1
                           \end{array}
\right)$. Then clearly $\delta$ is a contraction and it is easy see that $\alpha=\delta\beta$ and $\beta=\delta\alpha,$ as required.

\end{example}

\begin{example}
Consider
 $\alpha=\left( \begin{array}{cccc}
                            \{1,2\} & \{3,4\} & \{5,6\}  \\
                            4 & 3  & 2
                           \end{array}
\right)$ and $\beta=\left( \begin{array}{cccc}
                            \{1,2, 3\} & \{4, 5\}& 6  \\
                            {4} & {3}  & 2
                           \end{array}
\right)$  in $\mathcal{CT}_{6}$. Notice that the partitions $\textbf{Ker}~\gamma_{1}=\{\{1,2,3\}, \{4\},\{5\}, \{6\}\}$ and $\textbf{Ker}~\gamma_{2}=\{\{1,2\}, \{3\},\{4\}, \{5,6\}\}$ (for some $\gamma_{1}$, $\gamma_{2}$ in $\mathcal{CT}_{6}$) are refined partitions of the partitions $\textbf{Ker}~\alpha$ and $\textbf{Ker}~\beta$, respectively. Moreover, they are admissible since each  has a convex transversal. The maps  defined by $\gamma_{1}=\left( \begin{array}{ccccc}
                            \{1,2,3,\}& 4 &  5 & 6 \\
                            2  & 3 &4 & 5
                           \end{array}
\right)$ and $\gamma_{2}=\left( \begin{array}{ccccc}
                            \{1,2\}& 3 &  4 & \{5,6\} \\
                            2  & 3 &4 & 5
                           \end{array}
\right)$ are contractions in $\mathcal{CT}_{6}$  and it is easy check that $\alpha=\gamma_{1}\beta$ and $\beta=\gamma_{2}\alpha.$

\end{example}

 Now, let us deduce a characterization of  regular elements in $\mathcal{CT}_{n}$ from Lemma~\eqref{bb}.

\begin{corollary}\label{regt} Let $\alpha \in \mathcal{CT}_{n}$.   Then $\alpha$ is regular if and only if $\textbf{Ker}~\alpha$ has a convex transversal.
\end{corollary}
\begin{proof} Let $\alpha\in \mathcal{CT}_{n}$. (Notice that $\mathcal{CT}_{n}$ is a subsemigroup of $\mathcal{CP}_{n}$, thus $\alpha\in \mathcal{CP}_{n}.$)
Suppose also that $\alpha$ is regular, then by Theorem~\eqref{reeg}, there exists an admissible transversal $T_{\alpha}=\{t_{1},\ldots, t_{p}\}$ such that, $t_{i}\mapsto x_{i}$ ($i=1,2,\ldots,p$) is an isometry. Since $\alpha\in \mathcal{CT}_{n}$ then by Lemma~\eqref{tc} $\im~\alpha$ is convex and the fact that the map $t_{i}\mapsto x_{i}$ or $t_{i}\mapsto x_{p-i+1}$ ($i=1,2,\ldots,p$) is an isometry implies $T_{\alpha}$ is convex.

Conversely, suppose $T_{\alpha}=\{t_{1},\ldots, t_{p}\}$ is convex. Then by Lemma~\eqref{tc}, $T_{\alpha}\alpha$ is convex. This means $\im~\alpha$ is convex. And the map $t_{i}\mapsto x_{i}$ ($i=1,2,\ldots,p$) is an isometry. Thus, by Theorem~\eqref{reeg} $\alpha$ is regular.

\end{proof}
In view of the above Corollary and Corollary~\eqref{tt} we give as remarks the following:
\begin{remark}
 \begin{itemize}
\item[(i)] Notice that if $\textbf{Ker}~\alpha$ and $\textbf{Ker}~\beta$ both have convex transversals, then both $\alpha$ and $\beta$ are regular and so the result in [\cite{garbac}, Theorem 3.1, (i), (iii)] holds only for regular elements in $\mathcal{CT}_{n}$.

 \item[(ii)] A transversal $T_{\alpha}$ of $\alpha\in \mathcal{CT}_{n}$ is admissible if and only if it is convex.

 \end{itemize}
\end{remark}

\begin{corollary}\label{Regt} Let $\alpha \in\mathcal{ORCT}_{n}$.  Then $\alpha$ is regular if and only if $\min A_{p}-x_{p}=\max A_{1}-x_{1}=d$ and $A_{i}=\{x_{i}+d\}$ or $\min A_{p}-x_{1}=\max A_{1}-x_{p}=d$ and $A_{i}=\{x_{p-i+1}+d\}$, for $i=2,\ldots,p-1$.
\end{corollary}
\begin{proof} The proof is the same as that of Corollary 4.2 in \cite{mmz}.
\end{proof}

\begin{corollary} Let $\alpha\in\mathcal{OCT}_{n}$.  Then $\alpha$ is regular if and only if $\min A_{p}-x_{p}=\max A_{1}-x_{1}=d$ and $A_{i}=\{x_{i}+d\}$, for $i=2,\ldots,p-1$.
\end{corollary}
\begin{proof} This follows from Corollary~\eqref{Regt}.
\end{proof}

\begin{remark} Product of idempotents in $\mathcal{CT}_{n}$ is not necessarily an  idempotent. For example consider $\alpha=\left(\begin{array}{cc}
                                       1 & \{2,3,4\}   \\
                                       1 & 2
                                     \end{array}
 \right)$ and $\beta=\left(\begin{array}{cc}
                                       \{1,3\} & \{2,4\}   \\
                                       3 & 2
                                     \end{array}
 \right)\in \mathcal{CT}_{3}$. Then $\alpha\beta=\left(\begin{array}{cc}
                                       1 & \{2,3,4\}   \\
                                       3 & 2
                                     \end{array}
 \right)$ is not an idempotent.
\end{remark}

 A semigroup $S$ is said to be \emph{left abundant} (resp., \emph{right abundant}) if each $\mathcal{{L}}^{*}-class$ (resp., $\mathcal{{R}}^{*}-class$) contains an idempotent and  it is said to be \emph{abundant} if $\mathcal{{L}}^{*}-class$ and $\mathcal{{R}}^{*}-class$ of $S$ both contains an idempotent \cite{FOUN}. That is to say that the semigroup has plentiful supply of idempotents. Examples of abundant semigroups are the regular semigroups. The study of abundant semigroups was initiated in 1979 by John Fountain \cite{FOUN,FOUN2}. Since then a lot of work has been done by many authors on the study of abundant semigroups. The starred Green's relations for the semigroups $\mathcal{CT}_{n}$ and $\mathcal{OCT}_{n}$ were characterized by Garba \emph{et al.} \cite{garbac} and curiously they did not show whether they are abundant or not.  We are now going to show that the semigroup $\mathcal{CT}_{n}$ and its subsemigroups $\mathcal{ORCT}_{n}$ and $\mathcal{OCT}_{n}$ are left abundant but not right abundant, in general.

The following theorem gives characterizations of starred Green's relations from \cite{garbac}.
\begin{theorem}[\cite{garbac}, Theorem 4.1]\label{starr} Let $\alpha,\beta$ be elements in $S\in \{\mathcal{CT}_{n}, \mathcal{OCT}_{n} \}$. Then we have the following:
\begin{itemize}
\item[(i)] $(\alpha,\beta)\in \mathcal{L}^{*}$ if and only if $\im~\alpha=\im~\beta$;
\item[(ii)] $(\alpha,\beta)\in \mathcal{R}^{*}$ if and only if $ker~\alpha=ker~\beta$;
\item[(iii)] $(\alpha,\beta)\in \mathcal{H}^{*}$ if and only if $\im~\alpha=\im~\beta$ and $ker~\alpha=ker~\beta$;
\item[(iv)] $(\alpha,\beta)\in \mathcal{D}^{*}$ if and only if $|\im~\alpha|=|\im~\beta|$.
\end{itemize}

\end{theorem}

We now prove the following lemmas:

  \begin{lemma} Let $S\in\{\mathcal{CT}_{n}, \mathcal{OCT}_{n}, \mathcal{ORCT}_{n}\}$. Then $S$ is left abundant.
 \end{lemma}

 \begin{proof}
 Let $\alpha\in S$ and $L^{*}_{\alpha}$ be an $\mathcal{{L}}^{*}-class$ of $\alpha$ in $S$, where $\alpha=\left(\begin{array}{cccc}
                                                                            A_{1} & A_{2} & \ldots & A_{p} \\
                                                                            x_{1} & x_{2} & \ldots & x_{p}
                                                                          \end{array}
\right)$ ($1\leq p\leq n$). Define   $\gamma=\left(\begin{array}{cccc}
                                                                            B_{1} & B_{2} & \ldots & B_{p} \\
                                                                            x_{1} & x_{2} & \ldots & x_{p}
                                                                          \end{array}
\right)$, where $\textbf{Ker}~\gamma=\{B_{1}<\ldots<B_{p}\}$ and $x_{i}\in B_{i}$ ($1\leq p\leq n$). Then $\gamma\in S$ and since the blocks are stationary we have $\gamma^{2}=\gamma$. Moreover,  observe that $\im~\alpha=\im~\gamma$, therefore by Theorem~\eqref{starr}(i), $\alpha \mathcal{L}^{*}\gamma$, which means that $\gamma\in L^{*}_\alpha$. This completes the proof.

 \end{proof}
\begin{lemma}Let $S\in\{\mathcal{CT}_{n}, \mathcal{ORCT}_{n}, \mathcal{OCT}_{n}\}$. Then for $ n\geq 4$, $S$ is  not right abundant.
\end{lemma}

\begin{proof} Let $n=4$ and consider $\alpha=\left(\begin{array}{ccc}
                                                                            1 & \{2,3\} & 4  \\
                                                                            1 & 2 &  3
                                                                          \end{array}
\right)$. Clearly $\alpha$ is in $S$ and $$R^{*}_\alpha=\left\{   \left(\begin{array}{ccc}
                                                                            1 & \{2,3\} & 4  \\
                                                                            1 & 2 &  3
                                                                          \end{array}
\right),\left(\begin{array}{ccc}
                                                                            1 & \{2,3\} & 4  \\
                                                                            3 & 2 &  1
                                                                          \end{array}
\right), \left(\begin{array}{ccc}
                                                                            1 & \{2,3\} & 4  \\
                                                                            2 & 3 &  4
                                                                          \end{array}
\right), \left(\begin{array}{ccc}
                                                                            1 & \{2,3\} & 4  \\
                                                                            4 & 3 &  2
                                                                          \end{array}
\right)   \right\},$$ \noindent which has no idempotent element.
\end{proof}

 \begin{remark} Let $S\in\{\mathcal{CT}_{n}, \mathcal{ORCT}_{n}, \mathcal{OCT}_{n}\}$. Then for $1\leq n\leq 3$, $S$ is right abundant.
\end{remark}

\section{ Regular elements of $\mathcal{CT}_{n}$}
A semigroup $S$ is said to be \emph{orthodox} if $E(S)$ is a subsemigroup of $S$. An orthodox semigroup is said to be \emph{$\mathcal{R}$-unipotent} (resp., \emph{$\mathcal{L}$-unipotent}) if every $\mathcal{R}$-class (resp., every $\mathcal{L}$-class) has a unique idempotent (see \cite{ph} Exercise 1.2.19). If an orthodox semigroup $S$ is both $\mathcal{R}$ and $\mathcal{L}$-unipotent then $S$ is an inverse semigroup. For  basic concept we refer the reader to \cite{howie3,hall3,ven}. Lets us recall  that an element $\alpha\in\mathcal{CT}_{n} $ is regular if and only if $\textbf{Ker}~\alpha$ has convex transversal. Further,  we denote $Reg(S)$ to be collection of all regular element of $S$. If $A$ is a subset of $S$ then $\langle A\rangle$ denotes the semigroup generated by $A$. Moreover, $\langle A\rangle=A$ if and only if $A$ is a subsemigroup of $S$ and if $\langle A\rangle=S$ then $A$ is said to generate $S$. A regular element in $\mathcal{CP}_{n}$ is said to be \emph{strongly regular} if and only if $T_{\alpha}$ is convex. Concept of strongly regular in  $\mathcal{CP}_{n}$ has been introduced in \cite{az}. We give the following remark.

\begin{remark}\label{rem} In $\mathcal{CT}_{n}$ an element is regular if and only if $T_{\alpha}$ is convex. Thus, the concepts of strongly regular and regularity coincide in $\mathcal{CT}_{n}$.
\end{remark}
The following results from \cite{az} will be found to be useful in our subsequent discussion.
\begin{lemma}[\cite{az}, Lemma 3.3]\label{mm} Let $\alpha$ be an idempotent element in $SReg(\mathcal{ORCP}_{n})$.  Then $\alpha$ can be expressed as $\alpha=\left(\begin{array}{cccccc}
                                                                            A_{1} & a+2 & a+3 & \ldots & a+p-1 & A_{p} \\
                                                                            a+1 & a+2 & a+3 & \ldots & a+p-1 & a+ p
                                                                          \end{array}
\right)$,  where $a+1=\max A_{1}$, $a+p=\min A_{p}$.
\end{lemma}

\begin{lemma}[\cite{az}, Lemma 3.4]\label{kkk} Let $\epsilon=\left(\begin{array}{ccccc}
                                                                            A_{1} & a+2 &  \ldots & a+p-1 & A_{p} \\
                                                                            a+{1} & a+2 &  \ldots & a+p-1 & a+ p
                                                                          \end{array}
\right)$ and \\$\tau=\left(\begin{array}{ccccc}
                                                                            B_{1} & b+2 &  \ldots & b+s-1 & B_{s} \\
                                                                            b+{1} & b+2 &  \ldots & b+s-1 & b+ s
                                                                          \end{array}
\right)$ be two idempotents elements in $SReg(\mathcal{ORCP}_{n})$. Then $\epsilon\tau$ is  strongly regular.

\end{lemma}

The next result is from \cite{Hall}.
\begin{proposition} [\cite{Hall}, Proposition 1.]\label{idp1}   Let $S$ be an arbitrary semigroup. Then the following are equivalent:
\begin{itemize}
\item[(i)] For all idempotents $e$ and $f$ of $S$, the element $ef$ is regular;
\item[(ii)] $Reg(S)$ is regular subsemigroup;
\item[(iii)]    $\langle E(S) \rangle$  is a regular semigroup.
\end{itemize}
\end{proposition}

 Then we have the following lemma.

\begin{lemma}\label{idp2} Let

$$\epsilon=\left(\begin{array}{cccc}
                                                                   A_{1} & A_{2} & \ldots & A_{p} \\
                                                                   t_{1} & t_{2} & \ldots & t_{p}
                                                                 \end{array}
 \right),~ \tau=\left(\begin{array}{cccc}
                                                                   B_{1} & B_{2} & \ldots & B_{p} \\
                                                                   t_{1}^{'} & t_{2}^{'} & \ldots & t_{p}^{'}
                                                                 \end{array}
 \right) \in E(\mathcal{CT}_{n}), ~(p\leq n).$$

\noindent Then $\epsilon \tau$ is regular.
\end{lemma}
\begin{proof}

Let   $\epsilon=\left(\begin{array}{cccc}
                                                                   A_{1} & A_{2} & \ldots & A_{p} \\
                                                                   t_{1} & t_{2} & \ldots & t_{p}
                                                                 \end{array}
 \right) $  and      $\tau=\left(\begin{array}{cccc}
                                                                   B_{1} & B_{2} & \ldots & B_{p} \\
                                                                   t_{1}^{'} & t_{2}^{'} & \ldots & t_{p}^{'}
                                                                 \end{array}
 \right)$  ($p\leq n$)  be idempotents in $\mathcal{CT}_{n}$.  Notice that $T_{\epsilon}=\{ t_{1}, t_{2}, \ldots, t_{p}\}$  is convex and $T_{\epsilon}\epsilon=\{ t_{1}\epsilon, t_{2}\epsilon, \ldots, t_{p}\epsilon\}$ is convex and that $T_{\epsilon}\epsilon\tau=\{ t_{1}\epsilon\tau, t_{2}\epsilon\tau, \ldots, t_{p}\epsilon\tau\}$ whose elements are not necessarily distinct but is nevertheless convex. Moreover, any convex transversal of $\textbf{Ker}~\epsilon\tau$ is isometric to some subset of $T_{\epsilon}$.  Hence $\textbf{Ker}~\epsilon\tau$ must contain at least one convex transversal, as such by Lemma~\eqref{tc}  $\im~\epsilon\tau$ is convex, and by Theorem~\eqref{Regt} $\epsilon\tau$ is regular.
\end{proof}

As a consequence of Proposition\eqref{idp1} and Lemma~\eqref{idp2} we have the following:

\begin{corollary} Let $S=\mathcal{CT}_{n}$. Then we have:
\begin{itemize}
 \item[(i)] $ Reg (S)$ is a regular subsemigroup of $S$;
 \item[(ii)]  $\langle E(S) \rangle$  is a regular subsemigroup of $S$.
\end{itemize}
\end{corollary}

\section{Semigroup of order preserving or order reversing full contractions: $\mathcal{ORCT}_{n}$}

In view of Remark\eqref{rem} we deduce the following  corollaries from Lemmas~\eqref{mm} and \eqref{kkk}, respectively.

\begin{corollary}\label{mmm} Let $\epsilon$ be an idempotent element in $Reg(\mathcal{ORCT}_{n})$. Then $\alpha$ can be expressed as $$\epsilon=\left(\begin{array}{cccccc}
                                                                            \{1,\ldots,i\} & i+1 & i+2 & \ldots & i+j-1 & \{i+j,\ldots,n\} \\
                                                                            i & i+1 & i+2 & \ldots & i+j-1 & i+ j
                                                                          \end{array}
\right).$$
\end{corollary}

\begin{corollary}\label{kk} Let $\epsilon$ and $\tau$ be two idempotents in $Reg(\mathcal{ORCP}_{n})$. Then the product $\epsilon\tau$ is also an idempotent.

\end{corollary}

Now as a consequence of Proposition~\eqref{idp1} and Corollary~\eqref{kk}, we readily have:

\begin{theorem}\label{a} Let $S=\mathcal{ORCT}_{n}$. Then we have the following:

\begin{itemize} \item[(i)] $Reg(S)$ is a regular subsemigroup of $S$;
\item[(ii)] $\langle E(\mathcal{ORCT}_{n})\rangle=E(\mathcal{ORCT}_{n})$;
\item[(iii)] $Reg(\mathcal{ORCT}_{n})$ is orthodox.
\end{itemize}
\end{theorem}

Next, we are now going to show that $Reg(\mathcal{ORCT}_{n})$ is indeed a special orthodox semigroup. However, first we establish the following lemma:

\begin{lemma}\label{zx1}
Let $\alpha=\left(\begin{array}{cccccc}
                                                                            A_{1} & a+1 & a+2 & \ldots & a+p-2 & A_{p} \\
                                                                            x & x+ 1 & x+2 & \ldots & x+p-2 & x+ p-1
                                                                          \end{array}
\right)$ be an element of $Reg(\mathcal{ORCT}_{n})$. Then ${L}_{\alpha}$ contains a unique idempotent.
\end{lemma}

\begin{proof}
Notice that the map $$\epsilon=\left(\begin{array}{cccccc}
                                                                            \{1, \ldots, x\} & x+1 & x+2 & \ldots & x+p-2 & \{x+p-1, \ldots, n\} \\
                                                                            x & x+ 1 & x+2 & \ldots & x+p-2 & x+ p-1
                                                                          \end{array}
\right)$$
\noindent is obviously a unique idempotent in ${L}_{\alpha}$.
\end{proof}

Thus we have  from the above lemma and by definition that:
\begin{corollary}\label{zx}  $Reg(\mathcal{ORCT}_{n})$ is an $\mathcal{L}$-unipotent semigroup.
\end{corollary}

Notice that $Reg(\mathcal{ORCT}_{n})$ is not an $\mathcal{R}$-unipotent semigroup. To see this, consider
 $\alpha=\left(\begin{array}{c}
                                       [n]  \\
                                       x
                                     \end{array}
 \right)$~and ~ $\alpha^{'}=\left(\begin{array}{c}
                                       [n]  \\
                                       x^{'}
                                     \end{array}
 \right)$ where $x^{'}\neq x\in [n]$ are distinct idempotents in $R_{\alpha}$.

\begin{remark} The results proved in this section for the semigroup $\mathcal{ORCT}_{n}$  hold  when $\mathcal{ORCT}_{n}$ is replaced with $\mathcal{OCT}_{n}$.
\end{remark}

\section{Rees quotients of $Reg(\mathcal{ORCT}_{n})$ }
In this section we construct a Rees factor semigroup from $Reg(\mathcal{ORCT}_{n})$ and show that it is an inverse semigroup. For $n\geq p\geq 2$, let

\begin{eqnarray} K(n,p)=\{\alpha \in Reg(\mathcal{ORCT}_{n}):|\im~\alpha|\leq p\}
\end{eqnarray}
\noindent be the two-sided ideal of $Reg(\mathcal{ORCT}_{n})$ consisting of all elements of height less than or equal to $p$. Further, let
\begin{eqnarray} Q_p(n)=K(n,p)/K(n,p-1)
\end{eqnarray}
\noindent be the Rees factor or quotient semigroup on the two-sided ideal $K(n,p)$. The product of two elements in $Q_p(n)$ is zero if its height is less than $p$, otherwise it is as in $Reg(\mathcal{ORCT}_{n})$.

Immediately, we have the following lemma.

\begin{theorem}
   The semigroup $Q_p(n)$ is an inverse semigroup.
\end{theorem}
\begin{proof} It is clear from Lemma~\eqref{zx1} that  $Q_{p}(n)$ is $\mathcal{L}$-unipotent. To show it is $\mathcal{R}$-unipotent, let $\alpha\in  Q_{p}(n)$, where $$\alpha=\left(\begin{array}{cccccc}
                                                                            \{1, \ldots, a\} & a+1 & a+2 & \ldots & a+p-2 & \{a+p-1, \ldots, n\} \\
                                                                            x & x+ 1 & x+2 & \ldots & x+p-2 & x+ p-1
                                                                          \end{array}
\right)~(p\geq 2).$$  Consider $R_{\alpha}-$class. Notice that the map defined as $$\epsilon=\left(\begin{array}{cccccc}
                                                                            \{1, \ldots, a\} & a+1 & a+2 & \ldots & a+p-2 & \{a+p-1, \ldots, n\} \\
                                                                            a & a+ 1 & a+2 & \ldots & a+p-2 & a+ p-1
                                                                          \end{array}
\right)$$ is in $Q_{p}(n)$ and clearly $ker~\alpha=ker~\epsilon$, thus $\epsilon\in \mathcal{R}_{\alpha} $ by Corollary~\eqref{tt}(ii). Furthermore, notice that the blocks of $\epsilon$ are stationary, i. e.,  $\epsilon$ is an idempotent and obviously unique in $\mathcal{R}_{\alpha}$. Hence $Q_{p}(n)$ is  $\mathcal{R}$-unipotent and hence $Q_{p}(n)$ is an inverse semigroup, as required.
\end{proof}

\begin{remark}The results proved in this section for the semigroup $\mathcal{ORCT}_{n}$  hold  when $\mathcal{ORCT}_{n}$ is replaced with $\mathcal{OCT}_{n}$.
\end{remark}

\noindent {\bf Acknowledgements.} The second named author would like to thank Bayero University and TET Fund for financial support. He would also like to thank The Petroleum Institute, Khalifa University of Science and Technology for hospitality during his 3-month research visit to the institution.

\end{document}